\documentclass[11pt,reqno]{amsart}

\usepackage{amsmath,amssymb}
\usepackage{microtype}
\usepackage[hidelinks]{hyperref}
\allowdisplaybreaks
\numberwithin{equation}{section}

\newtheorem{theorem}{Theorem}[section]
\newtheorem{proposition}[theorem]{Proposition}
\newtheorem{lemma}[theorem]{Lemma}
\newtheorem{corollary}[theorem]{Corollary}
\theoremstyle{remark}
\newtheorem{remark}[theorem]{Remark}

\newcommand{\E}{\mathbb E}
\newcommand{\e}{\mathrm e}
\newcommand{\dd}{\,\mathrm d}
\newcommand{\MP}{\mathrm{MP}}
\renewcommand\le{\leqslant}
\renewcommand\ge{\geqslant}

\title[Fractional Moments and Shape Transitions in the Unitary Case]{Dimension Monotonicity in Laguerre Ensembles I: Fractional Moments and Shape Transitions in the Unitary Case}
\author{Ondrej Hutn\'{i}k}
\address{Institute of Mathematics, Faculty of Science, Pavol Jozef \v{S}af\'{a}rik University in Ko\v{s}ice, Jesenn\'{a} 5, 040 01 Ko\v{s}ice, Slovakia}
\email{ondrej.hutnik@upjs.sk}
\hypersetup{pdftitle={Dimension Monotonicity in Laguerre Ensembles I: Fractional Moments and Shape Transitions in the Unitary Case},pdfauthor={Ondrej Hutnik}}

\begin{document}

\begin{abstract}
Let $W_{N,N+\lambda}$ have the Laguerre unitary distribution with size $N$ and
real shape $\lambda\ge0$.  For $s>0$, we consider the normalized moment
\[
 C_{s,\lambda}(N)=N^{-s-1}\E[\operatorname{Tr}W_{N,N+\lambda}^{s}]
\]
and its dimension decrement
$\Gamma_{N,s,\lambda}=C_{s,\lambda}(N)-C_{s,\lambda}(N+1)$.
Iterating the Laguerre moment recurrence separates this decrement into a square
source and a nonnegative shape source.  The square source gives the complete
finite-dimensional sign diagram: $C_{s,0}(N)$ decreases for $0<s<1$ and
$s>2$, increases for $1<s<2$, and is constant for $s\in\{1,2\}$.  For every
$s>0$ and every $N$, the decrement is strictly increasing in $\lambda$.
The same decomposition determines the critical shrinking-shape scales:
$N^{-2s}$ for $0<s<1/2$, $(\log N)/N$ for $s=1/2$, and $N^{-1}$ for
$s>1/2$.  In the convex range $1<s<2$, where the two sources have opposite
signs, the transition occurs when $N\lambda_N$ is of order one, with critical
constant
\[
 \tau_s^*=\frac{s(s-1)(2-s)}{6(2s-1)}.
\]
In the convex range, both crossings are unique for every finite $N$, and we
determine their locations to second order.  At $s=1/2$ we also obtain a
bounded-shape two-term expansion, which supplies the unitary estimates used in
the companion orthogonal paper.
\end{abstract}

\maketitle

\section{Introduction}

\subsection{The model and the finite-dimensional question}
For real $\lambda\ge0$, let $W_{N,N+\lambda}$ denote the size-$N$ Laguerre
unitary ensemble (LUE).  Its ordered eigenvalues have density proportional to
\cite[Chapter~3]{Forrester}
\[
 \prod_{i=1}^N x_i^\lambda \e^{-x_i}
 \prod_{1\le i<j\le N}(x_j-x_i)^2,
 \qquad 0<x_1<\cdots<x_N.
\]
When $\lambda$ is a nonnegative integer, this is the eigenvalue law of
$X_{N,\lambda}^{\mathbb C}(X_{N,\lambda}^{\mathbb C})^*$, where
$X_{N,\lambda}^{\mathbb C}$ is an $N\times(N+\lambda)$ matrix of independent
standard complex Gaussian entries with density $\pi^{-1}\e^{-|z|^2}$.
Although the Laguerre density is defined for every real $\lambda>-1$, we work
with $\lambda\ge0$ throughout.

For $s>0$ set
\begin{equation}\label{eq:def-C-Gamma}
 \begin{aligned}
 Q_{s,\lambda}(N)&:=\E[\operatorname{Tr}W_{N,N+\lambda}^{s}],\\ 
 C_{s,\lambda}(N)&:=\frac{Q_{s,\lambda}(N)}{N^{s+1}},\\
 \Gamma_{N,s,\lambda}&:=C_{s,\lambda}(N)-C_{s,\lambda}(N+1).
 \end{aligned}
\end{equation}
The normalization removes the leading $N^{s+1}$ growth of the moment.  Thus
$\Gamma_{N,s,\lambda}>0$ means that the normalized moment decreases when the
dimension is increased from $N$ to $N+1$ while the excess shape $\lambda$ is
kept fixed.

For every fixed $s>0$ and $\lambda\ge0$, the moment form of the
Marchenko--Pastur law gives \cite{MarchenkoPastur,CundenEtAl}
\begin{equation}\label{eq:MP-limit}
 C_{s,\lambda}(N)\longrightarrow
 m_s:=\int_0^4x^s\,\mu_{\MP}(\dd x)
 =\frac{4^s\Gamma(s+1/2)}{\sqrt\pi\,\Gamma(s+2)}.
\end{equation}
The limit alone does not determine whether $C_{s,\lambda}(N)$ approaches
$m_s$ from above or below, nor whether the next step in dimension is upward or
downward.  These are the finite-dimensional questions studied here.

\subsection{Square shape and dependence on the shape parameter}
For integer $\lambda$, the half moment $s=1/2$ is the normalized expected
nuclear norm of $X_{N,\lambda}^{\mathbb C}$.  Its dimension monotonicity arose
in Gaussian rounding for the little Grothendieck problem over the unitary group
\cite[Conjecture~8]{BandeiraKennedySinger}.  Abreu reduced the complex square problem to Laguerre integral identities
\cite{Abreu}, and the resulting monotonicity statement was proved in the
first arXiv version of the present work~\cite{HutnikComplexOriginal}.  Baslingker and Dan later obtained the
half-moment monotonicity for every real $\lambda\ge0$ from the closed
Laguerre moment recurrence \cite{BaslingkerDan,CundenEtAl}.  Abreu and Patil
have recently sharpened the square case $(s,\lambda)=(1/2,0)$ with explicit
finite-dimensional bounds and complete asymptotic expansions
\cite{AbreuPatil}.

The closed recurrence is therefore an existing input, and the half-moment
monotonicity for real shape is already known.  The new results here concern
the two-parameter consequences of that recurrence: the finite-dimensional sign
diagram in the moment order, strict shape monotonicity of the decrement for
every $s>0$, the three shrinking-shape scales, and the finite-dimensional and
second-order crossings in $1<s<2$.  The square half-moment case remains the
special slice on which Abreu--Patil obtain sharper quantitative information.

The present paper treats the full two-parameter family $(s,\lambda)$.  At
square shape, the sign is determined completely:
\[
 \Gamma_{N,s,0}
 \begin{cases}
 >0,&0<s<1\ \text{or}\ s>2,\\
 <0,&1<s<2,\\
 =0,&s=1,2.
 \end{cases}
\]
This holds for every $N\ge1$.  Thus, the normalized square moments decrease in
the concave range $0<s<1$, increase in the convex range $1<s<2$, and decrease
again for $s>2$.

The dependence of the decrement on the shape is also monotone.  For every
$s>0$ and every $N\ge1$,
\[
 \partial_\lambda\Gamma_{N,s,\lambda}>0,
 \qquad \lambda\ge0.
\]
Consequently, the two decreasing square ranges remain decreasing for every
$\lambda\ge0$.  In the convex range $1<s<2$, by contrast, a positive shape can change the sign of the total decrement, which leads to the transition described below.

\subsection{Three critical shape scales}
The proof starts from an exact formula obtained by iterating the Laguerre
moment recurrence.  It writes $\Gamma_{N,s,\lambda}$ as the sum of two terms:
a square source, present already at $\lambda=0$, and a shape source, which is
nonnegative and is strictly positive for $\lambda>0$.  The square source is a
fourth-order remainder for the power function.  Its sign gives the square
diagram above, while its summability determines the size of a shape
perturbation that can change the leading finite-size correction.

There are three regimes.  If $\lambda_N\to0$, the relevant scales are
\[
 \lambda_N\asymp N^{-2s},\qquad 0<s<\frac12,
\]
\[
 \lambda_N\asymp\frac{\log N}{N},\qquad s=\frac12,
\]
and
\[
 \lambda_N\asymp N^{-1},\qquad s>\frac12.
\]
The logarithm at $s=1/2$ is caused by the harmonic leading term of the square
source.  In this sense the average-singular-value problem is the borderline
case between a summable and a nonsummable square correction.

\subsection{The convex transition}
The range $1<s<2$ is the only one in which the square and shape sources have
opposite signs.  Here the relevant scaling is $N\lambda_N\to\tau$.  The
critical constant for the decrement is
\begin{equation*}
 \tau_s^*=\frac{s(s-1)(2-s)}{6(2s-1)}.
\end{equation*}
The level and decrement do not cross at the same value.  Their finite-$N$
zeros, denoted by $\lambda^{\mathrm{lev}}_{N,s}$ and
$\lambda^{\mathrm{dec}}_{N,s}$, are unique and satisfy
\[
 \lambda^{\mathrm{lev}}_{N,s}
 =\frac{\tau_s^*}{2N}+o(N^{-2}),
 \qquad
 \lambda^{\mathrm{dec}}_{N,s}
 =\frac{\tau_s^*}{N}-\frac{\tau_s^*}{2N^2}+o(N^{-2}).
\]
These asymptotics imply that, for all sufficiently large $N$, there is an
interval in which the normalized moment is already above the Marchenko--Pastur
limit but is still increasing with the dimension.  The level crossing and the
local monotonicity crossing are therefore different finite-dimensional
quantities.

\subsection{The half moment and recent square results}
At $s=1/2$ we write
$C_\lambda(N)=C_{1/2,\lambda}(N)$ and
$\Gamma_{N,\lambda}=\Gamma_{N,1/2,\lambda}$.  The general theory gives,
uniformly for $\lambda$ in a bounded interval,
\begin{align*}
 \Gamma_{N,\lambda}
 &=\frac{2\lambda}{\pi N^2}
 +\frac{1-4\lambda^2}{8\pi}\frac{\log N}{N^3}
 +O(N^{-3}),\\
 C_\lambda(N)-\frac8{3\pi}
 &=\frac{2\lambda}{\pi N}
 +\frac{1-4\lambda^2}{16\pi}\frac{\log N}{N^2}
 +O(N^{-2}).
\end{align*}
At $\lambda=0$, Abreu and Patil obtain sharper square-specific asymptotics,
including the constant terms.  The purpose of the formulas above is different:
they are uniform in the shape and exhibit the coefficient $1-4\lambda^2$,
which vanishes at $\lambda=1/2$.  The corresponding positive-shape and shape-derivative bounds are then used in the companion orthogonal paper (Paper~II)~\cite{HutnikLOE}.

\subsection{Proof strategy and organization}
The finite-dimensional argument starts from the closed dimension recurrence of
Cunden, Mezzadri, O'Connell and Simm.  After normalization, one iteration
already separates the term present at square shape from the term created by
$\lambda>0$; iterating in the dimension gives the two-source formula used
throughout the paper.  The square source is controlled by a one-dimensional
Taylor remainder.  The shape dependence requires two further ingredients:
positive association of the ordered LUE eigenvalues, which gives monotonicity
of $C_{s,\lambda}(N)$, and a relative derivative estimate obtained by
normalizing the same recurrence by its one-dimensional solution.

Section~\ref{sec:exact} develops these finite-dimensional tools and proves
strict shape monotonicity of the decrement.  Section~\ref{sec:square} applies
the decomposition at square shape and on the two sign-positive moment ranges.
Section~\ref{sec:shape} compares the two sources asymptotically, then treats the
finite-dimensional and second-order convex transitions.  Finally,
Section~\ref{sec:half} specializes the results to the half moment and records
the estimates used in Paper~II.

\section{The recurrence and the two-source formula}\label{sec:exact}

\subsection{The recurrence and its iteration}

The starting point is the closed dimension recurrence obtained from the
Laguerre moment theory of Cunden, Mezzadri, O'Connell and Simm
\cite[Section~4.2, after (4.16)]{CundenEtAl}.  Equivalently, it follows from
the recurrence recorded by Baslingker and Dan \cite[(1.3)]{BaslingkerDan}:
\begin{equation}\label{eq:Q-recurrence}
 Q_{s,\lambda}(N+1)
 =\left(2+\frac{s(s+1)}{N(N+\lambda)}\right)Q_{s,\lambda}(N)
 -Q_{s,\lambda}(N-1),
 \quad N\ge1,
\end{equation}
with the convention $Q_{s,\lambda}(0):=0$.  It is valid for every real $s>0$ in the parameter range used here.  Thus, all moments entering the argument are ordinary finite LUE integrals; no analytic continuation in $s$ is needed.

The normalization in \eqref{eq:def-C-Gamma} removes the leading power of
$N$.  What remains at square shape is a fourth-order remainder of the power
function.  To keep that remainder visible, let $r=s+1$ and set
\[
 F_s(x):=(1+x)^r+(1-x)^r-2-s(s+1)x^2,
 \qquad 0\le x\le1.
\]
With this notation, \eqref{eq:Q-recurrence} becomes, for $N\ge2$,
\begin{equation}\label{eq:Gamma-step}
\begin{aligned}
 \Gamma_{N,s,\lambda}
 &={}\left[
 \left(\frac{N}{N+1}\right)^rF_s(1/N)
 +\frac{s(s+1)\lambda N^{s-1}}
 {(N+\lambda)(N+1)^r}\right]C_{s,\lambda}(N)\\
 &\quad+\left(\frac{N-1}{N+1}\right)^r
 \Gamma_{N-1,s,\lambda}.
\end{aligned}
\end{equation}
For $N=1$, the convention $Q_{s,\lambda}(0)=0$ gives the same formula
without the last term.  This is the base case for the iteration below.

The two terms in the bracket have different roles: the first is already
present at $\lambda=0$, whereas the second vanishes there and is positive for
$\lambda>0$.  Iteration preserves this separation.

\begin{proposition}[Two-source decomposition]\label{prop:exact}
For every $s>0$, $\lambda\ge0$, and $N\ge1$,
\begin{equation*}
\begin{aligned}
 \Gamma_{N,s,\lambda}
 &=\frac{1}{[N(N+1)]^{s+1}}
   \sum_{k=1}^N k^{2s+2}F_s(1/k)C_{s,\lambda}(k)\\
 &\quad+
 \frac{s(s+1)\lambda}{[N(N+1)]^{s+1}}
   \sum_{k=1}^N\frac{k^{2s}}{k+\lambda}C_{s,\lambda}(k).
\end{aligned}
\end{equation*}
\end{proposition}

We call the first and second lines the \textit{square source} and the \textit{shape source}, respectively.

\begin{proof}
Iterate \eqref{eq:Gamma-step} from the $N=1$ base case.  The propagation
factor telescopes as
\[
 \prod_{j=k+1}^N\left(\frac{j-1}{j+1}\right)^{s+1}
 =\left(\frac{k(k+1)}{N(N+1)}\right)^{s+1}.
\]
Substituting the bracket in \eqref{eq:Gamma-step} and simplifying gives the
two displayed sums.
\end{proof}

\subsection{The sign of the square source}

The following integral representation determines the sign of the square source for every positive moment order.

\begin{lemma}[Sign of the square source]\label{lem:F-sign}
For every $s>0$ and $0<x\le1$,
\begin{equation}\label{eq:F-series}
 F_s(x)=2\sum_{j=2}^\infty\binom{s+1}{2j}x^{2j}.
\end{equation}
For $0<x<1$ it also has the integral representation
\begin{equation}\label{eq:F-integral}
 F_s(x)=\frac{s(s+1)(s-1)(s-2)}{6}
 \int_0^x(x-t)^3
 \bigl((1+t)^{s-3}+(1-t)^{s-3}\bigr)\,\dd t.
\end{equation}
Moreover,
\begin{align*}
 F_s(x)&>0, \qquad 0<s<1\ \text{or}\ s>2,\\
 F_s(x)&<0, \qquad 1<s<2,\\
 F_1(x)&=F_2(x)=0.
\end{align*}
As $x\downarrow0$,
\begin{equation}\label{eq:F-leading}
 F_s(x)=\eta_sx^4+O_s(x^6),
 \quad
 \eta_s:=2\binom{s+1}{4}
 =\frac{s(s+1)(s-1)(s-2)}{12}.
\end{equation}
\end{lemma}

\begin{proof}
The binomial identity gives \eqref{eq:F-series}; for noninteger $s$ the
series is absolutely convergent also at $x=1$.  Applying Taylor's formula
with integral remainder to $t^{s+1}$ about $t=1$, once at $1+x$ and once
at $1-x$, gives \eqref{eq:F-integral}.  Its integrand is positive, so its
prefactor gives all the asserted strict signs for $0<x<1$; continuity gives
$x=1$.  At $s\in\{1,2\}$ the remainder vanishes identically.  Finally, the first
term of \eqref{eq:F-series} gives \eqref{eq:F-leading}.
\end{proof}

\subsection{Shape monotonicity of the moment}

To compare the two sources as $\lambda$ varies, we first need monotonicity of
the normalized moment itself.  The same argument also upgrades the fixed-shape
Marchenko--Pastur limit to a uniform limit on bounded shape intervals.

\begin{lemma}[Shape monotonicity and uniform Marchenko--Pastur limit]
\label{lem:shape-uniform}
Fix $s>0$.  For every $N\ge1$, the map
$\lambda\mapsto C_{s,\lambda}(N)$ is strictly increasing on $[0,\infty)$.
Consequently, for every finite $\Lambda\ge0$, we have
\begin{equation}\label{eq:uniform-MP}
 \sup_{0\le\lambda\le\Lambda}
 \bigl|C_{s,\lambda}(N)-m_s\bigr|\longrightarrow0.
\end{equation}
In particular, $C_{s,\lambda}(N)$ is bounded uniformly in
$N\ge1$ and $0\le\lambda\le\Lambda$.
\end{lemma}

\begin{proof}
\smallskip\noindent\emph{Differentiating in the shape.}
Fix $\lambda\ge0$ and let
\[
 \mathcal W_N
 :=\{(x_1,\ldots,x_N)\in(0,\infty)^N:
        x_1<\cdots<x_N\}.
\]
Write the ordered LUE density as
\[
 f_{N,\lambda}(x)
 =\frac{1}{Z_N(\lambda)}
   \prod_{i=1}^N x_i^\lambda \e^{-x_i}
   \prod_{1\le i<j\le N}(x_j-x_i)^2,
 \qquad x\in\mathcal W_N.
\]
Set
\[
 H_s(x):=\sum_{i=1}^N x_i^s,
 \qquad
 \ell(x):=\sum_{i=1}^N\log x_i,
\]
and let $Y$ have density $f_{N,\lambda}$.  In this proof, $\E_\lambda$
and $\operatorname{Cov}_\lambda$ refer to expectation and covariance with
respect to this law.  The parameter $\lambda$ enters through the exponential
tilt $\exp(\lambda\ell(x))$.  Differentiating the normalized integral gives
\begin{align}
 \partial_\lambda C_{s,\lambda}(N)
 &=N^{-s-1}
   \left(
     \E_\lambda[H_s(Y)\ell(Y)]
     -\E_\lambda H_s(Y)\,\E_\lambda \ell(Y)
   \right) \notag\\
 &=N^{-s-1}\operatorname{Cov}_\lambda
   \bigl(H_s(Y),\ell(Y)\bigr).
 \label{eq:shape-cov}
\end{align}
On compact shape intervals this differentiation is justified by the usual
hard-edge integrability of the logarithmic factors and the exponential decay
at infinity.

\smallskip\noindent\emph{Positivity of the covariance.}
We use association of the ordered eigenvalues.  The chamber $\mathcal W_N$ is
a sublattice, and the one-variable factors $x_i^\lambda\e^{-x_i}$ are
modular.  For a Vandermonde factor $h(u,v)=(v-u)^2$, the only nontrivial
ordering is, after interchanging the two points if necessary,
$x_i\le y_i<y_j\le x_j$.  In that case
\[
 (y_j-x_i)(x_j-y_i)-(x_j-x_i)(y_j-y_i)
 =(y_i-x_i)(x_j-y_j)\ge0.
\]
All factors are nonnegative, so squaring proves the MTP$_2$ inequality for
$h$.  After extending the density by zero outside $\mathcal W_N$, the full
density is therefore MTP$_2$.  Karlin--Rinott association
\cite{KarlinRinott}, with truncation for the unbounded functions, gives
\[
 \operatorname{Cov}_\lambda(Y_i^s,\log Y_j)\ge0
 \qquad(1\le i,j\le N).
\]
Strictness comes from the diagonal terms.  If $Y_i'$ is an independent copy
of the nondegenerate coordinate $Y_i$, then
\[
 2\operatorname{Cov}_\lambda(Y_i^s,\log Y_i)
 =\E\!\left[(Y_i^s-(Y_i')^s)(\log Y_i-\log Y_i')\right]>0,
\]
because both scalar functions are strictly increasing.  Hence
\[
 \operatorname{Cov}_\lambda(H_s(Y),\ell(Y))
 =\sum_{i,j=1}^N
 \operatorname{Cov}_\lambda(Y_i^s,\log Y_j)>0,
\]
and \eqref{eq:shape-cov} proves strict shape monotonicity.

\smallskip\noindent\emph{Uniformity on bounded shape intervals.}
Let $\Lambda<\infty$.  The monotonicity just proved gives
\[
 C_{s,0}(N)\le C_{s,\lambda}(N)\le C_{s,\Lambda}(N),
 \qquad 0\le\lambda\le\Lambda.
\]
Both endpoint sequences converge to $m_s$ by \eqref{eq:MP-limit}.  Therefore
\[
 \sup_{0\le\lambda\le\Lambda}
 \bigl|C_{s,\lambda}(N)-m_s\bigr|
 \le
 \max\left\{
   |C_{s,0}(N)-m_s|,
   |C_{s,\Lambda}(N)-m_s|
 \right\}
 \longrightarrow0.
\]
This is \eqref{eq:uniform-MP}.  It also gives the stated uniform boundedness
for all sufficiently large $N$; the finitely many remaining dimensions are
bounded on $[0,\Lambda]$ by continuity.
\end{proof}

\subsection{A relative shape-derivative bound}

Moment monotonicity alone is enough when the square source is nonnegative.  In
the convex range $1<s<2$, however, some scalar brackets below are negative;
then the increase of $C_{s,\lambda}(N)$ itself has to be controlled.  The next
lemma compares its relative shape derivative with the one-dimensional moment.

\begin{lemma}[Relative shape derivative]
For $s>0$ put
\[
 q_{s,\lambda}:=Q_{s,\lambda}(1)
 =\frac{\Gamma(\lambda+s+1)}{\Gamma(\lambda+1)},
 \qquad
 \delta_s(\lambda):=\partial_\lambda\log q_{s,\lambda}.
\]
Then, for every $N\ge1$ and $\lambda\ge0$,
\begin{equation}\label{eq:relative-shape-derivative}
 0<
 \frac{\partial_\lambda C_{s,\lambda}(N)}{C_{s,\lambda}(N)}
 \le \delta_s(\lambda).
\end{equation}
The upper inequality is strict for $N\ge2$.  If $s>1$, then
\begin{equation}\label{eq:delta-bound}
 \delta_s(\lambda)<\frac{s}{\lambda+1}.
\end{equation}
\end{lemma}

\begin{proof}
\smallskip\noindent\emph{The one-dimensional normalization.}
We divide by $q_{s,\lambda}$ because its shape dependence is explicit, leaving
a recurrence factor whose dependence on $\lambda$ has a definite sign.  Set
\[
 U_N(\lambda):=\frac{Q_{s,\lambda}(N)}{q_{s,\lambda}},
 \qquad
 v_n(\lambda):=\frac{s(s+1)}{n(n+\lambda)}.
\]
Then $U_0=0$, $U_1=1$, and the recurrence becomes
\[
 U_{n+1}=(2+v_n)U_n-U_{n-1}.
\]
For $\Delta U_n:=U_n-U_{n-1}$ this is
\[
 \Delta U_1=1,\qquad \Delta U_{n+1}=\Delta U_n+v_nU_n.
\]
Thus $\Delta U_n>0$ and $U_n>0$.  Since $v_n'(\lambda)<0$,
\[
 (\Delta U_{n+1})'=(\Delta U_n)'+v_n'U_n+v_nU_n'.
\]
Starting from $(\Delta U_1)'=U_1'=0$, induction gives
$(\Delta U_n)'\le0$ and $U_n'\le0$; moreover $U_n'<0$ for $n\ge2$.  Indeed,
once the two derivatives are nonpositive at step $n$, both additional terms
on the right are nonpositive at step $n+1$, and $v_n'U_n<0$ gives strictness.
Since
\[
 C_{s,\lambda}(N)
 =N^{-s-1}q_{s,\lambda}U_N(\lambda),
\]
we obtain
\[
 \partial_\lambda\log C_{s,\lambda}(N)
 =\delta_s(\lambda)+\partial_\lambda\log U_N(\lambda)
 \le\delta_s(\lambda).
\]
The strict lower bound was proved in Lemma~\ref{lem:shape-uniform}, and the
upper bound is strict for $N\ge2$ because then $U_N'<0$.

\smallskip\noindent\emph{Bounding the one-dimensional derivative.}
For $a=\lambda+1$, the logarithmic derivative of the Gamma ratio is
\[
 \delta_s(\lambda)
 =\int_0^\infty
 \e^{-at}\frac{1-\e^{-st}}{1-\e^{-t}}\,\dd t.
\]
If $s>1$, strict convexity of $y\mapsto y^s$ gives
$1-\e^{-st}<s(1-\e^{-t})$ for $t>0$, and therefore
\[
 \delta_s(\lambda)
 <s\int_0^\infty \e^{-(\lambda+1)t}\,\dd t
 =\frac{s}{\lambda+1}.
\] This completes the proof.
\end{proof}

Before differentiating the decrement, we record two consequences that will
be used in the convex range.  If $s\ge1$, the power-mean inequality followed
by Jensen's inequality and the exact first moment
$Q_{1,\lambda}(N)=N(N+\lambda)$ (also obtained directly from
\eqref{eq:Q-recurrence} at $s=1$) give
\[
 C_{s,\lambda}(N)\ge\left(1+\frac{\lambda}{N}\right)^s\ge1.
\]
Further, for $k\ge1$ set
\[
 G_{k,s}(\lambda)
 :=k^2F_s(1/k)+\frac{s(s+1)\lambda}{k+\lambda}.
\]
With the same notation, the exact decomposition takes the form
\begin{equation}\label{eq:global-source-bracket}
 \Gamma_{N,s,\lambda}
 =\frac{1}{[N(N+1)]^{s+1}}
 \sum_{k=1}^Nk^{2s}C_{s,\lambda}(k)G_{k,s}(\lambda).
\end{equation}
This form isolates the only difficulty in differentiating the decrement.  If
$G_{k,s}(\lambda)\ge0$, both factors in a summand vary in the favorable
direction.  If $G_{k,s}(\lambda)<0$, the increase of
$C_{s,\lambda}(k)$ contributes with the wrong sign, and the relative derivative
bound above is needed to compare it with $G_{k,s}'(\lambda)>0$.

\subsection{Shape monotonicity of the decrement}

\begin{theorem}[Strict shape monotonicity of the decrement]\label{thm:global-shape-monotonicity}
For every $s>0$ and $N\ge1$, the map
$\lambda\mapsto\Gamma_{N,s,\lambda}$ is strictly increasing on
$[0,\infty)$.  If $1<s<2$, then, more quantitatively,
\begin{align}
 \partial_\lambda\Gamma_{N,s,\lambda}
 &>
 \frac{s(s-1)}
 {(\lambda+1)[N(N+1)]^{s+1}}
 \sum_{k=1}^N
 \frac{k^{2s}}{(k+\lambda)^2}C_{s,\lambda}(k)
 \label{eq:convex-shape-reserve}\\
 &\ge
 \frac{s(s-1)}
 {(\lambda+1)[N(N+1)]^{s+1}}
 \sum_{k=1}^N\frac{k^{2s}}{(k+\lambda)^2}.
 \label{eq:convex-shape-reserve-explicit}
\end{align}
\end{theorem}

\begin{proof}
\smallskip\noindent\emph{The sign-positive ranges.}
Since
\[
 G_{k,s}'(\lambda)=\frac{s(s+1)k}{(k+\lambda)^2}>0,
\]
if $0<s\le1$ or $s\ge2$, Lemmas~\ref{lem:F-sign} and
\ref{lem:shape-uniform} show directly that every summand in
\eqref{eq:global-source-bracket} has strictly positive shape derivative.

\smallskip\noindent\emph{The convex range.}
Now let $1<s<2$.  Here $F_s<0$, and a summand with
$G_{k,s}(\lambda)<0$ is the only possible obstruction: the factor
$C_{s,\lambda}(k)$ increases with $\lambda$, but it multiplies a negative
bracket.  A simple uniform estimate for $F_s$ will be enough to control this
case:
\begin{equation}\label{eq:F-convex-bound}
 0<-F_s(x)<2x^4,\qquad 0<x\le1.
\end{equation}
Indeed, every coefficient in \eqref{eq:F-series} is negative in this
range, so
\[
 -F_s(x)\le x^4[-F_s(1)].
\]
Moreover,
\[
 -F_s(1)=2+s(s+1)-2^{s+1}<2.
\]
For the last inequality, observe that $s+1<2^s$ on $(1,2)$: the function $f(s)=2^s-s-1$ satisfies $f(1)=0$, $f'(1)=2\log2-1>0$, and $f''>0$. Multiplying by $s<2$ gives $s(s+1)<2^{s+1}$.

Now fix $k$.  If $G_{k,s}(\lambda)\ge0$, then
\[
 \partial_\lambda\bigl(C_{s,\lambda}(k)G_{k,s}(\lambda)\bigr)
 \ge C_{s,\lambda}(k)G_{k,s}'(\lambda)>0.
\]
It remains to consider $G_{k,s}(\lambda)<0$.  In this case
\eqref{eq:relative-shape-derivative} and \eqref{eq:delta-bound} give
\[
0<
\frac{\partial_\lambda C_{s,\lambda}(k)}
     {C_{s,\lambda}(k)}
\le \delta_s(\lambda)
<\frac{s}{\lambda+1}.
\]
Multiplication by the negative bracket reverses the last inequality, so the
derivative of the summand is bounded from below by
\begin{align*}
 \frac{\partial_\lambda
 \bigl(C_{s,\lambda}(k)G_{k,s}(\lambda)\bigr)}
 {C_{s,\lambda}(k)}
 &=
 G_{k,s}'(\lambda)
 +G_{k,s}(\lambda)
 \frac{\partial_\lambda C_{s,\lambda}(k)}
 {C_{s,\lambda}(k)}\\
 &>
 G_{k,s}'(\lambda)
 +\frac{s}{\lambda+1}G_{k,s}(\lambda)=\frac{sP_{k,s}(\lambda)}
 {(\lambda+1)(k+\lambda)^2},
\end{align*}
where the remaining question has been reduced to the sign of the quadratic
polynomial
\begin{align*}
 P_{k,s}(\lambda)
 &=
 \bigl(k^4F_s(1/k)+(s+1)k\bigr)+
 \bigl(2k^3F_s(1/k)+(s+1)^2k\bigr)\lambda\\
 &\quad+
 \bigl(k^2F_s(1/k)+s(s+1)\bigr)\lambda^2.
\end{align*}
By \eqref{eq:F-convex-bound},
$F_s(1/k)>-2/k^4$, and therefore, coefficientwise,
\[
 P_{k,s}(\lambda)
 >
 \bigl((s+1)k-2\bigr)
 +\left((s+1)^2k-\frac4k\right)\lambda
 +\left(s(s+1)-\frac2{k^2}\right)\lambda^2.
\]
All three coefficients on the right are strictly positive for
$1<s<2$ and $k\ge1$, because
$(s+1)k>2$, $(s+1)^2k^2>4$, and $s(s+1)>2$.  Thus, the coarse estimate
$F_s(x)>-2x^4$ is already strong enough to give
$P_{k,s}(\lambda)>0$ for every $\lambda\ge0$.  This settles the only case in
which the derivative of a summand did not have an immediate sign.

\smallskip\noindent\emph{The quantitative bound.}
The same estimates give the stated lower bound for the derivative.  In the case
$G_{k,s}(\lambda)<0$ we have
$P_{k,s}(\lambda)>-2+(s+1)k\ge s-1$, while if
$G_{k,s}(\lambda)\ge0$, then
\[
 G_{k,s}'(\lambda)
 >\frac{s(s-1)}{(\lambda+1)(k+\lambda)^2}.
\]
Thus, in both cases,
\[
 \partial_\lambda
 \bigl(C_{s,\lambda}(k)G_{k,s}(\lambda)\bigr)
 >
 \frac{s(s-1)}{(\lambda+1)(k+\lambda)^2}
 C_{s,\lambda}(k).
\]
Substitution in \eqref{eq:global-source-bracket} proves
\eqref{eq:convex-shape-reserve}.  Finally, the convex-moment lower bound recorded above gives
$C_{s,\lambda}(k)\ge1$, which proves
\eqref{eq:convex-shape-reserve-explicit}.

\end{proof}

\section{The square sign diagram}\label{sec:square}

\subsection{Square shape}

At square shape, the second source in Proposition~\ref{prop:exact} vanishes.
Lemma~\ref{lem:F-sign} therefore gives a complete finite-dimensional sign
theorem for every positive moment order.

\begin{theorem}[Exact square sign diagram]\label{thm:square-sign}
For every $N\ge1$, we have
\begin{align*}
 \Gamma_{N,s,0} & > 0, \quad 0<s<1\ \text{or}\ s>2, \\
 \Gamma_{N,s,0} & < 0, \quad 1<s<2, \\ \Gamma_{N,1,0}& = \Gamma_{N,2,0}=0.
\end{align*}
Equivalently, $N\mapsto C_{s,0}(N)$ is strictly decreasing for
$0<s<1$ and $s>2$, strictly increasing for $1<s<2$, and constant at
$s=1,2$.  In particular,
\begin{align*}
 C_{s,0}(N)&>m_s, \quad 0<s<1\ \text{or}\ s>2,\\
 C_{s,0}(N)&<m_s, \quad 1<s<2,\\
 C_{1,0}(N)&=m_1=1, \\ C_{2,0}(N)&=m_2=2.
\end{align*}
\end{theorem}

\begin{proof}
Every $C_{s,0}(k)$ is positive, so the sign assertions follow from Proposition~\ref{prop:exact} and Lemma~\ref{lem:F-sign}.  The comparison
with $m_s$ follows by letting the monotone sequence converge in \eqref{eq:MP-limit}.  At $s=1$ and $s=2$ one can also use the exact identities
$Q_{1,0}(N)=N^2$ and $Q_{2,0}(N)=2N^3$.
\end{proof}

\begin{corollary}[Monotonicity on the sign-positive bands]
\label{cor:concave-all-shapes}
For every $s\in(0,1)\cup(2,\infty)$, every real $\lambda\ge0$, and
every $N\ge1$, we get
\[
 \Gamma_{N,s,\lambda}>0.
\]
Consequently $C_{s,\lambda}(N)>m_s$ and
$N\mapsto C_{s,\lambda}(N)$ decreases strictly to $m_s$.
At $s=1$ and $s=2$ the same statements hold for $\lambda>0$, while the square sequences are constant.
\end{corollary}

\begin{proof}
For $s\in(0,1)\cup(2,\infty)$, both sources in
Proposition~\ref{prop:exact} are nonnegative and the square source is
strictly positive.  At $s\in\{1,2\}$ the square source vanishes and the shape
source is strictly positive when $\lambda>0$.
\end{proof}

\subsection{Quantitative shape bounds}

The sign-positive ranges need no cancellation argument.  Keeping only the
shape source gives explicit bounds that will later be specialized to the half
moment.

\begin{proposition}[Quantitative shape bounds]
\label{prop:shape-derivative-reserve}
Let $s\in(0,1]\cup[2,\infty)$.  Then for every $N\ge1$ and
$\lambda\ge0$,
\begin{equation}\label{eq:general-shape-derivative-reserve}
 \partial_\lambda\Gamma_{N,s,\lambda}
 \ge
 \frac{s(s+1)m_s}{[N(N+1)]^{s+1}}
 \sum_{k=1}^N\frac{k^{2s+1}}{(k+\lambda)^2}.
\end{equation}
The inequality is strict for $s\in(0,1)\cup(2,\infty)$.  Moreover, for $s\in(0,1]\cup[2,\infty)$ and $\lambda>0$, we have
\begin{equation}\label{eq:general-positive-shape-reserve}
 \Gamma_{N,s,\lambda} >  \frac{s(s+1)m_s\lambda}{[N(N+1)]^{s+1}} \sum_{k=1}^N\frac{k^{2s}}{k+\lambda}.
\end{equation}
For $s\in(0,1)\cup(2,\infty)$, the last strict inequality also holds at $\lambda=0$, with zero right-hand side.
\end{proposition}

\begin{proof}
\smallskip\noindent\emph{Derivative bound.}
Differentiate the exact formula in Proposition~\ref{prop:exact}.  By
Lemma~\ref{lem:shape-uniform}, $\partial_\lambda C_{s,\lambda}(k)\ge0$.
Since $F_s(1/k)\ge0$ in the stated range, the derivative of the square source
is nonnegative.  The derivative of the shape source is
\begin{align*}
 &\frac{s(s+1)}{[N(N+1)]^{s+1}}
 \sum_{k=1}^N
 \left\{
 \frac{k^{2s+1}}{(k+\lambda)^2}C_{s,\lambda}(k)
 +\frac{\lambda k^{2s}}{k+\lambda}
   \partial_\lambda C_{s,\lambda}(k)
 \right\}.
\end{align*}
Corollary~\ref{cor:concave-all-shapes} gives
$C_{s,\lambda}(k)\ge m_s$, proving
\eqref{eq:general-shape-derivative-reserve}.  The derivative is strictly positive because its first displayed summand is positive.  On the open sign-positive bands, $C_{s,\lambda}(k)>m_s$, so the derivative
bound is strict.

\smallskip\noindent\emph{Positive-shape lower bound.}
For the second estimate, drop the nonnegative square source from
Proposition~\ref{prop:exact} and use
$C_{s,\lambda}(k)>m_s$ when $\lambda>0$.  This gives
\eqref{eq:general-positive-shape-reserve}.  On the open bands, the square source itself is strictly positive, which handles $\lambda=0$.
\end{proof}

\section{Critical shapes and the convex transition}\label{sec:shape}

\subsection{Asymptotics of the two sources}

The critical shape scale is determined by comparing the size of the square
source with the universal $\lambda N^{-2}$ contribution of the shape source.
The square source changes its summability at $s=1/2$, so its leading
coefficient has a different form below and above that value.  For
$0<s<1/2$ define
\begin{equation}\label{eq:K-def}
 \mathcal K_s:=\sum_{k=1}^\infty
 k^{2s+2}F_s(1/k)C_{s,0}(k),
\end{equation}
and for $s>1/2$ put
\begin{equation}\label{eq:b-def}
 \mathfrak b_s:=\frac{\eta_sm_s}{2s-1}
 =\frac{2\binom{s+1}{4}m_s}{2s-1}.
\end{equation}
By Lemma~\ref{lem:F-sign}, every summand in \eqref{eq:K-def} is positive;
by \eqref{eq:F-leading} it is asymptotic to
$\eta_sm_sk^{2s-2}$.  Since $2s-2<-1$ in this range, the series converges and
$\mathcal K_s>0$.  For the shape source set
\begin{equation}\label{eq:a-def}
 a_s:=\frac{s+1}{2}m_s.
\end{equation}
The two terms in Proposition~\ref{prop:exact} then have the following
asymptotic forms.

\begin{lemma}[Asymptotics of the two sources]\label{lem:source-asymptotics}
Fix $s>0$ and $\Lambda<\infty$. Then
\begin{equation}\label{eq:shape-source-asymptotic}
 \frac{s(s+1)\lambda}{[N(N+1)]^{s+1}}
 \sum_{k=1}^N\frac{k^{2s}}{k+\lambda}C_{s,\lambda}(k)
 =\frac{\lambda}{N^2}\bigl(a_s+o_\Lambda(1)\bigr)
\end{equation}
uniformly for $0\le\lambda\le\Lambda$.  For the square source, along every
sequence $\lambda_N\to0$, we have
\begin{align}
 &[N(N+1)]^{-s-1}
 \sum_{k=1}^Nk^{2s+2}F_s(1/k)C_{s,\lambda_N}(k)
 \notag\\
 &\qquad\sim \mathcal K_sN^{-2s-2}, &&0<s<1/2,
 \label{eq:source-low}\\
 &\qquad\sim \frac1{8\pi}\frac{\log N}{N^3}, &&s=1/2,
 \label{eq:source-half}\\
 &\qquad\sim \mathfrak b_sN^{-3}, &&s>1/2,\quad s\notin\{1,2\}.
 \label{eq:source-high}
\end{align}
At $s\in\{1,2\}$, the square source is zero.
\end{lemma}

\begin{proof}
\smallskip\noindent\emph{The shape source.}
By Lemma~\ref{lem:shape-uniform}, $C_{s,\lambda}(k)\to m_s$ as
$k\to\infty$ uniformly for $0\le\lambda\le\Lambda$.  Also
$k/(k+\lambda)\to1$ uniformly on that shape interval.  Hence
\[
 \sum_{k=1}^N\frac{k^{2s}}{k+\lambda}C_{s,\lambda}(k)
 =\frac{m_s}{2s}N^{2s}+o_\Lambda(N^{2s}).
\]
Substitution gives \eqref{eq:shape-source-asymptotic}.

\smallskip\noindent\emph{The square source below the threshold.}
Let $\lambda_N\to0$ and choose $\Lambda_0<\infty$ such that $0\le\lambda_N\le\Lambda_0$ for all sufficiently large $N$.  By \eqref{eq:F-leading}, we obtain
\begin{equation}\label{eq:square-summand-expansion}
 k^{2s+2}F_s(1/k)
 =\eta_s k^{2s-2}+O_s(k^{2s-4}).
\end{equation}
Suppose first that $0<s<1/2$.  By the uniform boundedness in
Lemma~\ref{lem:shape-uniform},
\[
 \sup_{k\ge1,\ 0\le\lambda\le\Lambda_0}C_{s,\lambda}(k)<\infty.
\]
Together with \eqref{eq:square-summand-expansion}, this bounds the
absolute value of the $k$th square-source summand by
$O_{s,\Lambda_0}(k^{2s-2})$.  This majorant is summable because
$2s-2<-1$.  For every fixed $k$, continuity in the shape parameter gives
\[
 C_{s,\lambda_N}(k)\longrightarrow C_{s,0}(k).
\]
Applying dominated convergence to the summands extended by zero for
$k>N$, we obtain
\[
 \sum_{k=1}^N
 k^{2s+2}F_s(1/k)C_{s,\lambda_N}(k)
 \longrightarrow
 \sum_{k=1}^\infty
 k^{2s+2}F_s(1/k)C_{s,0}(k)
 =\mathcal K_s.
\]
Since
\[
 [N(N+1)]^{-s-1}\sim N^{-2s-2},
\]
this proves \eqref{eq:source-low}.

\smallskip\noindent\emph{The square source at and above the threshold.}
For $s\ge1/2$, Lemma~\ref{lem:shape-uniform} gives
\[
 C_{s,\lambda}(k)=m_s+o_{\Lambda_0}(1),
 \qquad k\to\infty,
\]
uniformly for $0\le\lambda\le\Lambda_0$.  Combining this with
\eqref{eq:square-summand-expansion} and the uniform boundedness from the
same lemma yields, for $s\notin\{1,2\}$,
\begin{equation}\label{eq:square-summand-uniform}
 k^{2s+2}F_s(1/k)C_{s,\lambda}(k)
 =\bigl(\eta_sm_s+o_{\Lambda_0}(1)\bigr)k^{2s-2},
\end{equation}
uniformly for $0\le\lambda\le\Lambda_0$. If $s=1/2$, the power in \eqref{eq:square-summand-uniform} is $-1$.
Therefore,
\begin{align*}
 \sum_{k=1}^N
 k^3F_{1/2}(1/k)C_{1/2,\lambda_N}(k)
 &=\eta_{1/2}m_{1/2}\log N+o(\log N)=\frac{\log N}{8\pi}+o(\log N).
\end{align*}
Since $[N(N+1)]^{-3/2}\sim N^{-3}$, this proves
\eqref{eq:source-half}.

If $s>1/2$ and $s\notin\{1,2\}$, then $2s-2>-1$, and the standard power-sum asymptotic applied to \eqref{eq:square-summand-uniform} gives
\begin{align*}
 \sum_{k=1}^N
 k^{2s+2}F_s(1/k)C_{s,\lambda_N}(k)
 &=\frac{\eta_sm_s}{2s-1}N^{2s-1}+o(N^{2s-1})\\
 &=\mathfrak b_sN^{2s-1}+o(N^{2s-1}).
\end{align*}
Multiplication by
$[N(N+1)]^{-s-1}\sim N^{-2s-2}$ proves
\eqref{eq:source-high}.  Finally, at $s\in\{1,2\}$ we have
$F_s\equiv0$, so the square source vanishes identically.
\end{proof}

\subsection{Three critical scales}

The comparison of these source orders gives all three critical scales at
once.

\begin{theorem}[Critical shape scales]\label{thm:critical-scales}
Let $\lambda_N\ge0$.
\begin{enumerate}
\item If $0<s<1/2$ and $N^{2s}\lambda_N\to\tau\in[0,\infty)$,
then
\begin{align*}
 N^{2s+2}\Gamma_{N,s,\lambda_N}&\longrightarrow \mathcal K_s+a_s\tau,
 \\
 N^{2s+1}\bigl(C_{s,\lambda_N}(N)-m_s\bigr)
 &\longrightarrow\frac{\mathcal K_s}{2s+1}+a_s\tau.
\end{align*}
\item If $s=1/2$ and
$N\lambda_N/\log N\to\tau\in[0,\infty)$, then
\begin{align*}
 \frac{N^3}{\log N}\Gamma_{N,1/2,\lambda_N}
 &\longrightarrow\frac1{8\pi}+\frac{2\tau}{\pi},
 \\
 \frac{N^2}{\log N}
 \left(C_{1/2,\lambda_N}(N)-\frac8{3\pi}\right)
 &\longrightarrow\frac1{16\pi}+\frac{2\tau}{\pi}.
\end{align*}
\item If $s>1/2$ and $N\lambda_N\to\tau\in[0,\infty)$,
then
\begin{align*}
 N^3\Gamma_{N,s,\lambda_N}&\longrightarrow \mathfrak b_s+a_s\tau,
 \\
 N^2\bigl(C_{s,\lambda_N}(N)-m_s\bigr)
 &\longrightarrow\frac{\mathfrak b_s}{2}+a_s\tau.
\end{align*}
\end{enumerate}
\end{theorem}

\begin{proof}
\smallskip\noindent\emph{The decrement.}
For each of the three scalings, the decrement limit follows directly from
Lemma~\ref{lem:source-asymptotics}: one substitutes the prescribed size of
$\lambda_N$ and compares the square-source term with
$\lambda_N a_sN^{-2}$.

\smallskip\noindent\emph{From the decrement to the level.}
The level requires summation in the dimension index.  Fix $N$ and, while
$j$ ranges from $N$ to infinity, keep the same parameter $\lambda_N$ fixed.
Since $C_{s,\lambda_N}(j)\to m_s$,
\begin{equation}\label{eq:telescoping-moving}
 C_{s,\lambda_N}(N)-m_s
 =\sum_{j=N}^\infty\Gamma_{j,s,\lambda_N}.
\end{equation}
For the shape source, \eqref{eq:shape-source-asymptotic} gives, uniformly on
bounded shape intervals,
\[
 \lambda j^{-2}(a_s+\varepsilon_{j,\lambda}),
 \qquad
 \sup_{0\le\lambda\le\Lambda}|\varepsilon_{j,\lambda}|\longrightarrow0.
\]
Its tail is therefore
$a_s\lambda_N/N+\lambda_No(N^{-1})$.

For the square source, the form of the tail depends on the same summability
threshold as before.  If $s\ge1/2$, the proof of
Lemma~\ref{lem:source-asymptotics} gives uniformly for $j\ge N$ the leading
term $(8\pi)^{-1}(\log j)j^{-3}$ at $s=1/2$ and
$\mathfrak b_sj^{-3}$ for $s>1/2$; at $s=1,2$ the term is identically zero.
If $0<s<1/2$, define
\[
 \mathcal K_s(\lambda):=\sum_{k=1}^\infty
 k^{2s+2}F_s(1/k)C_{s,\lambda}(k).
\]
The series is uniformly convergent on bounded shape intervals by
Lemma~\ref{lem:shape-uniform} and \eqref{eq:square-summand-expansion}.  Hence
$\mathcal K_s(\lambda_N)\to\mathcal K_s$, and the square source is
$j^{-2s-2}(\mathcal K_s+o(1))$ uniformly over the tail $j\ge N$.

\smallskip\noindent\emph{Summing the asymptotics.}
The preceding estimates also provide summable majorants: a constant times
$j^{-2s-2}$ for $s<1/2$, $(1+\log j)j^{-3}$ for $s=1/2$, and $j^{-3}$ for
$s>1/2$, together with $\lambda_Nj^{-2}$ for the shape source.  We may
therefore sum the source asymptotics in \eqref{eq:telescoping-moving}.  Using
\begin{align*}
 \sum_{j=N}^\infty j^{-2s-2}&\sim\frac{N^{-2s-1}}{2s+1},\\
 \sum_{j=N}^\infty\frac{\log j}{j^3}
 &\sim\frac{\log N}{2N^2},\\
 \sum_{j=N}^\infty j^{-3}&\sim\frac1{2N^2},
 \qquad
 \sum_{j=N}^\infty j^{-2}\sim\frac1N,
\end{align*}
we obtain the three stated level limits.
\end{proof}

Taking $\lambda_N\equiv0$ in Theorem~\ref{thm:critical-scales} recovers the
three square regimes: the decrement has scale $N^{-2s-2}$ for $s<1/2$,
$(\log N)N^{-3}$ at $s=1/2$, and $N^{-3}$ for $s>1/2$; the corresponding
level scales are $N^{-2s-1}$, $(\log N)N^{-2}$, and $N^{-2}$.  The limiting
coefficients are respectively
\[
 \mathcal K_s,\quad \frac1{8\pi},\quad \mathfrak b_s
 \qquad\text{for the decrement,}
\]
and
\[
 \frac{\mathcal K_s}{2s+1},\quad \frac1{16\pi},\quad \frac{\mathfrak b_s}{2}
 \qquad\text{for the level.}
\]
At $s\in\{1,2\}$ the square source vanishes identically.

\begin{remark}[Why the half moment is logarithmic]
The number $s=1/2$ is not singled out by the recurrence itself.  It is the
point where the leading square-source summand $\eta_sm_sk^{2s-2}$ becomes
harmonic.  The logarithm in the average-singular-value problem is therefore
a summability transition in the moment order.
\end{remark}

\subsection{Finite-dimensional crossings in the convex range}

For $1<s<2$, the square and shape sources have opposite signs.  Before refining the
asymptotic transition, we record the exact finite-dimensional crossing structure.

\begin{lemma}[One-source crossing thresholds]
Let $1<s<2$ and, for $k\ge1$, define
\begin{equation}\label{eq:rho-def}
 \rho_{k,s}:=
 \frac{-k^3F_s(1/k)}
 {s(s+1)+k^2F_s(1/k)}.
\end{equation}
Then $\rho_{k,s}>0$, the sequence $k\mapsto\rho_{k,s}$ is strictly
decreasing, and
\begin{equation}\label{eq:rho-asymptotic}
 \rho_{k,s}
 =\frac{(s-1)(2-s)}{12k}+O_s(k^{-3}).
\end{equation}
\end{lemma}

\begin{proof}
Put $r=s+1\in(2,3)$, $c_s:=s(s+1)$, and
\[
 \widetilde F_s(x):=(1+x)^r+(1-x)^r-2=F_s(x)+c_sx^2.
\]
Strict convexity gives $\widetilde F_s(x)>0$ for $0<x\le1$, whereas
Lemma~\ref{lem:F-sign} gives $F_s(x)<0$.  Thus, the denominator in
\eqref{eq:rho-def} is positive.

The binomial series in \eqref{eq:F-series} shows that
$-F_s(x)/x^3$ is strictly increasing on $(0,1]$, while
$\widetilde F_s(x)/x^2$ is positive and strictly decreasing: after division
by the indicated powers of $x$, every nonconstant coefficient has the
corresponding sign.  Consequently
\[
 \rho_s(x):=\frac{-F_s(x)}{x\widetilde F_s(x)}
 =\frac{-F_s(x)/x^3}{\widetilde F_s(x)/x^2}
\]
is strictly increasing.  Since $\rho_{k,s}=\rho_s(1/k)$, the map $k\mapsto\rho_{k,s}$ is strictly decreasing.  Finally,
\eqref{eq:F-leading} gives
\[
 \rho_s(x)=-\frac{\eta_s}{c_s}x+O_s(x^3)
 =\frac{(s-1)(2-s)}{12}x+O_s(x^3),
\]
which proves \eqref{eq:rho-asymptotic}.
\end{proof}

In particular, $\rho_{N,s}\asymp N^{-1}$, which already gives the $N^{-1}$ scale of the convex transition. The thresholds $\rho_{k,s}$ locate the sign changes of the individual scalar
brackets $G_{k,s}$.  Because they decrease with $k$, the first and last
brackets give an explicit interval for the zero of the full decrement.

\begin{proposition}[Finite-dimensional crossings]\label{prop:all-N-crossings}
Let $1<s<2$ and $N\ge1$.
\begin{enumerate}
\item The function $\lambda\mapsto C_{s,\lambda}(N)$ is strictly increasing
on $[0,\infty)$.  Hence there is a unique
$\lambda^{\mathrm{lev}}_{N,s}>0$ such that
\[
 C_{s,\lambda^{\mathrm{lev}}_{N,s}}(N)=m_s.
\]
\item If $N=1$, the decrement has its unique zero at
$\lambda=\rho_{1,s}$.  If $N\ge2$, then
\begin{align}
 \Gamma_{N,s,\lambda}&<0,
\qquad 0\le\lambda\le\rho_{N,s},\label{eq:rho-negative}\\
 \Gamma_{N,s,\lambda}&>0,
 \qquad \lambda\ge\rho_{1,s}.\label{eq:rho-positive}
\end{align}
For $N\ge2$, $\Gamma_{N,s,\lambda}$ has a unique zero
$\lambda^{\mathrm{dec}}_{N,s}\in(\rho_{N,s},\rho_{1,s})$.
\end{enumerate}
\end{proposition}

\begin{proof}
\smallskip\noindent\emph{The level crossing.}
Lemma~\ref{lem:shape-uniform} gives strict increase of
$\lambda\mapsto C_{s,\lambda}(N)$, while the convex-moment bound
\[
 C_{s,\lambda}(N)
 \ge \left(1+\frac{\lambda}{N}\right)^s
\]
shows that this quantity tends to infinity with $\lambda$.  At square shape,
Theorem~\ref{thm:square-sign} gives $C_{s,0}(N)<m_s$.  Continuity therefore
gives exactly one crossing of the level $m_s$.

\smallskip\noindent\emph{The decrement crossing.}
In \eqref{eq:global-source-bracket}, the $k$th scalar bracket
$G_{k,s}(\lambda)$ is strictly increasing and vanishes at $\rho_{k,s}$.  Since
$\rho_{k,s}$ decreases with $k$, every bracket is nonpositive when
$0\le\lambda\le\rho_{N,s}$ and every bracket is nonnegative when
$\lambda\ge\rho_{1,s}$.  This gives \eqref{eq:rho-negative} and
\eqref{eq:rho-positive}; for $N\ge2$ at least one bracket is strict at each
endpoint.  When $N=1$ there is only one summand, so the zero is exactly
$\rho_{1,s}$.  For $N\ge2$, existence follows between the two endpoint signs,
and strict shape monotonicity of the decrement from
Theorem~\ref{thm:global-shape-monotonicity} gives uniqueness.
\end{proof}

\subsection{Second-order location of the crossings}

At first order the level crossing is located at half the decrement threshold.
To distinguish the finite-$N$ locations beyond that order, we expand the
decrement once more and then sum it carefully enough to retain all terms that
can contribute at order $N^{-3}$ to the level.

\begin{theorem}[Second-order convex transition]
Fix $1<s<2$ and $T<\infty$.  Then
\begin{align}
 N^3\Gamma_{N,s,\tau/N}
 &=\mathfrak b_s+a_s\tau
 -\frac{\frac32\mathfrak b_s+a_s\tau}{N}+o_T(N^{-1}),
 \label{eq:second-order-dec}\\
 N^2\bigl(C_{s,\tau/N}(N)-m_s\bigr)
 &=\frac{\mathfrak b_s}{2}+a_s\tau+o_T(N^{-1}),
 \label{eq:second-order-level}
\end{align} uniformly for $0\le\tau\le T$. In particular, the $N^{-1}$ correction vanishes in the level expansion but
not in the decrement expansion.
\end{theorem}

\begin{proof}
\smallskip\noindent\emph{A preliminary level bound.}
We first need a uniform estimate that allows $C_{s,\lambda}(k)$ to be replaced
by $m_s$ in the second-order sums.  The exact decomposition gives
\begin{equation}\label{eq:convex-level-bound}
 |C_{s,\lambda}(k)-m_s|
 \le c_{s,\Lambda}\left(k^{-2}+\lambda k^{-1}\right)
\end{equation}
uniformly for $0\le\lambda\le\Lambda$ and $k\ge1$.  Indeed,
\eqref{eq:F-leading}, Lemma~\ref{lem:shape-uniform}, and
Proposition~\ref{prop:exact} imply
\[
 |\Gamma_{j,s,\lambda}|
 \le c_{s,\Lambda}\left(j^{-3}+\lambda j^{-2}\right),
\]
and summing from $j=k$ to infinity gives \eqref{eq:convex-level-bound}.

\smallskip\noindent\emph{The decrement.}
Put $p=2s-2\in(0,2)$ and set $\lambda=\tau/N$, with $0\le\tau\le T$.
For the square source, \eqref{eq:F-leading} and
\eqref{eq:convex-level-bound} allow us to replace the moment factor by
$m_s$ at the required precision:
\begin{align*}
 \sum_{k=1}^N k^{2s+2}F_s(1/k)C_{s,\lambda}(k)
 &=\eta_sm_s\sum_{k=1}^Nk^p+o_T(N^p)\\
 &=\mathfrak b_sN^{2s-1}+\frac{\eta_sm_s}{2}N^{2s-2}
   +o_T(N^{2s-2}).
\end{align*}
The replacement error is bounded by a constant times
\[
 \sum_{k=1}^Nk^{p-2}
 +\frac1N\sum_{k=1}^Nk^{p-1}=o(N^p),
\]
and the $O_s(k^{p-2})$ remainder in \eqref{eq:F-leading} has the same order.
Euler--Maclaurin gives the second displayed line.  Since
\[
 [N(N+1)]^{-s-1}
 =N^{-2s-2}\left(1-\frac{s+1}{N}+O_s(N^{-2})\right),
\]
the square source equals
\begin{equation}\label{eq:square-source-second-order}
 \mathfrak b_sN^{-3}-\frac{3\mathfrak b_s}{2}N^{-4}+o_T(N^{-4}).
\end{equation}
Here we used
$\eta_sm_s/2-(s+1)\mathfrak b_s=-3\mathfrak b_s/2$.

The shape source is treated in the same way.  Since $2s-1=p+1$, we get
\begin{align*}
 \sum_{k=1}^N\frac{k^{2s}}{k+\lambda}C_{s,\lambda}(k)
 &=m_s\sum_{k=1}^Nk^{2s-1}+o_T(N^{2s-1})\\
 &=\frac{m_s}{2s}N^{2s}
   +\frac{m_s}{2}N^{2s-1}+o_T(N^{2s-1}),
\end{align*}
and hence
\begin{equation}\label{eq:shape-source-second-order}
 a_s\tau N^{-3}-a_s\tau N^{-4}+o_T(N^{-4}).
\end{equation}
Adding \eqref{eq:square-source-second-order} and
\eqref{eq:shape-source-second-order} proves
\eqref{eq:second-order-dec}.

\smallskip\noindent\emph{The level.}
We now telescope the decrement over $j\ge N$.  The shape parameter remains
fixed at $\lambda=\tau/N$ while $j$ varies; it is not replaced by $\tau/j$.
Under the uniform restriction $0\le\lambda\le T/N$, the preceding calculation
gives
\begin{align*}
 \text{square source}
 &=\mathfrak b_sj^{-3}-\frac{3\mathfrak b_s}{2}j^{-4}+r_{j,N},\\
 \text{shape source}
 &=a_s\lambda j^{-2}-a_s\lambda j^{-3}+h_{j,N},
\end{align*}
where, uniformly for $0\le\tau\le T$,
\[
 |r_{j,N}|\le c_{s,T}\bigl(\chi_s(j)+\lambda j^{-4}\bigr),
 \qquad
 |h_{j,N}|\le c_{s,T}\bigl(\lambda j^{-4}+\lambda^2j^{-3}\bigr),
\]
and
\[
 \chi_s(j)=
 \begin{cases}
  j^{-2s-2},&1<s<3/2,\\
  (1+\log j)j^{-5},&s=3/2,\\
  j^{-5},&3/2<s<2.
 \end{cases}
\]
The three cases in $\chi_s$ come from the remainder of the power sum after
the first two Euler--Maclaurin terms: it is $O(1)$, $O(\log j)$, or
$O(j^{p-1})$ according as $p<1$, $p=1$, or $p>1$.  The replacement
$C_{s,\lambda}(k)-m_s$ contributes the same orders, together with
$O(\lambda j^p)$ from the shape-dependent part of
\eqref{eq:convex-level-bound}.  After multiplication by
$[j(j+1)]^{-s-1}$ this gives the stated bound for $r_{j,N}$.  Expanding
$k/(k+\lambda)=1+O(\lambda/k)$ gives the stated bound for $h_{j,N}$.
In particular,
\[
 \sum_{j=N}^\infty(|r_{j,N}|+|h_{j,N}|)=o_T(N^{-3}).
\]

It remains to sum the explicit terms.  We use
\begin{align*}
 \sum_{j=N}^\infty j^{-2}
 &=N^{-1}+\tfrac12N^{-2}+O(N^{-3}),\\
 \sum_{j=N}^\infty j^{-3}
 &=\tfrac12N^{-2}+\tfrac12N^{-3}+O(N^{-4}),\\
 \sum_{j=N}^\infty j^{-4}
 &=\tfrac13N^{-3}+O(N^{-4}).
\end{align*}
The cancellation at order $N^{-3}$ is worth making explicit.  For the square
part,
\[
 \mathfrak b_s\left(\frac12N^{-2}+\frac12N^{-3}\right)
 -\frac{3\mathfrak b_s}{2}\left(\frac13N^{-3}\right)
 =\frac{\mathfrak b_s}{2}N^{-2}+O(N^{-4}),
\]
while, since $\lambda=\tau/N$, the shape part satisfies
\[
 a_s\lambda\left(N^{-1}+\frac12N^{-2}\right)
 -a_s\lambda\left(\frac12N^{-2}\right)
 =a_s\tau N^{-2}+O_T(N^{-4}).
\]
Thus,
\[
 C_{s,\tau/N}(N)-m_s
 =\left(\frac{\mathfrak b_s}{2}+a_s\tau\right)N^{-2}
  +o_T(N^{-3}),
\]
which is \eqref{eq:second-order-level}.
\end{proof}

\begin{corollary}[Convex transition and crossing locations]
Let $1<s<2$, and define
\begin{equation}\label{eq:tau-star}
 \tau_s^*:=-\frac{\mathfrak b_s}{a_s}
 =\frac{s(s-1)(2-s)}{6(2s-1)}>0.
\end{equation}
If $N\lambda_N\to\tau$, then, away from the two boundary values:
\begin{center}
\begin{tabular}{c|c|c}
range of $\tau$ & $C_{s,\lambda_N}(N)-m_s$ &
$\Gamma_{N,s,\lambda_N}$\\ \hline
$0\le\tau<\tau_s^*/2$ & negative & negative\\
$\tau_s^*/2<\tau<\tau_s^*$ & positive & negative\\
$\tau>\tau_s^*$ & positive & positive.
\end{tabular}
\end{center}
Let $\lambda^{\mathrm{lev}}_{N,s}$ and $\lambda^{\mathrm{dec}}_{N,s}$ be the
unique level and decrement crossings from
Proposition~\ref{prop:all-N-crossings}.  Then
\begin{align}
 \lambda^{\mathrm{lev}}_{N,s}
 &=\frac{\tau_s^*}{2N}+o(N^{-2}),
 \label{eq:level-crossing-location}\\
 \lambda^{\mathrm{dec}}_{N,s}
 &=\frac{\tau_s^*}{N}-\frac{\tau_s^*}{2N^2}+o(N^{-2}).
 \label{eq:decrement-crossing-location}
\end{align}
Consequently, $\lambda^{\mathrm{lev}}_{N,s}<\lambda^{\mathrm{dec}}_{N,s}$ for all sufficiently large $N$, and the interval between them is a finite-size overshoot region: the normalized moment is already above its limiting value there but is still locally increasing with dimension.
\end{corollary}

\begin{proof}
For $1<s<2$ we have $\mathfrak b_s<0<a_s$.  The sign table follows from
Theorem~\ref{thm:critical-scales}, and \eqref{eq:tau-star} follows by
substituting \eqref{eq:a-def}, \eqref{eq:b-def}, and
\eqref{eq:F-leading}.

For the crossing locations, set
\[
 \tau_N^{\mathrm{lev}}:=N\lambda_{N,s}^{\mathrm{lev}},
 \qquad
 \tau_N^{\mathrm{dec}}:=N\lambda_{N,s}^{\mathrm{dec}}.
\]
The first-order limits in Theorem~\ref{thm:critical-scales}, together with
uniqueness, first place these two roots at
\[
 \tau_N^{\mathrm{lev}}\longrightarrow\frac{\tau_s^*}{2},
 \qquad
 \tau_N^{\mathrm{dec}}\longrightarrow\tau_s^*.
\]
Choose $T>\tau_s^*+1$.  For large $N$ both roots lie in $[0,T]$, so the
uniform second-order expansions may be evaluated at the roots themselves.
The level equation gives
\[
 0=\frac{\mathfrak b_s}{2}+a_s\tau_N^{\mathrm{lev}}+o(N^{-1}),
\]
and hence, since $\mathfrak b_s=-a_s\tau_s^*$,
\[
 \tau_N^{\mathrm{lev}}=\frac{\tau_s^*}{2}+o(N^{-1}).
\]
Similarly, the decrement equation gives
\[
 0=\mathfrak b_s+a_s\tau_N^{\mathrm{dec}}
 -\frac{\frac32\mathfrak b_s+a_s\tau_N^{\mathrm{dec}}}{N}
 +o(N^{-1}).
\]
Using $\tau_N^{\mathrm{dec}}\to\tau_s^*$ and
$\mathfrak b_s=-a_s\tau_s^*$ yields
\[
 \tau_N^{\mathrm{dec}}
 =\tau_s^*-\frac{\tau_s^*}{2N}+o(N^{-1}).
\]
Dividing the last two relations by $N$ proves
\eqref{eq:level-crossing-location} and
\eqref{eq:decrement-crossing-location}.
\end{proof}

\subsection{Fixed positive shape}

The shrinking-shape transitions disappear when $\lambda$ is fixed.  The shape
source then has order $N^{-2}$ and dominates the square correction in every
moment range.

\begin{corollary}[Fixed positive shape]
For every $s>0$ and every fixed $\lambda>0$, we have
\[
 \Gamma_{N,s,\lambda}\sim\frac{a_s\lambda}{N^2},
 \qquad
 C_{s,\lambda}(N)-m_s\sim\frac{a_s\lambda}{N}.
\]
In particular, even in the convex range $1<s<2$, every fixed positive
shape is eventually decreasing in dimension and approaches $m_s$ from
above.  In this convex range, the square sequence, by contrast, increases to
$m_s$ from below.
\end{corollary}

\begin{proof}
For fixed $\lambda>0$, the shape source has order $N^{-2}$ by
\eqref{eq:shape-source-asymptotic}.  The estimates used in the proof of
Lemma~\ref{lem:source-asymptotics} show that the square source is
$O(N^{-2s-2})$ for $s<1/2$, $O((\log N)N^{-3})$ at $s=1/2$, and
$O(N^{-3})$ for $s>1/2$.  This proves the decrement asymptotic.  The same
estimates give a summable $O(j^{-2})$ bound for the full decrement for all
sufficiently large $j$.  Since $\sum_{j=N}^\infty j^{-2}\sim N^{-1}$,
telescoping the decrement asymptotic gives the level asymptotic.
\end{proof}

\begin{remark}[The polynomial cases]
The two zero-square-source cases are completely explicit:
\[
 C_{1,\lambda}(N)=1+\frac\lambda N,\qquad
 \Gamma_{N,1,\lambda}=\frac\lambda{N(N+1)}.
\]
Moreover,
\begin{align*}
 C_{2,\lambda}(N)&=2+\frac{3\lambda}{N}+\frac{\lambda^2}{N^2},\\
 \Gamma_{N,2,\lambda}
 &=\frac{3\lambda}{N(N+1)}
   +\frac{\lambda^2(2N+1)}{N^2(N+1)^2}.
\end{align*}
Thus, $\mathfrak b_1=\mathfrak b_2=0$, $a_1=1$, $a_2=3$, and the third part of
Theorem~\ref{thm:critical-scales} includes both polynomial cases.
\end{remark}

\section{The half moment and average singular values}\label{sec:half}

\subsection{Unitary estimates used in Paper~II}

We now specialize to $s=1/2$.  This is the logarithmic critical moment order.
At integer shape,
$C_{1/2,\lambda}(N)=N^{-3/2}\E\|X_{N,\lambda}^{\mathbb C}\|_*$, so this is
also the average-singular-value statistic that motivated the original problem.
The exact two-source formula gives both the positive-shape estimates needed in
Paper~II and the bounded-shape expansion below.  Write
\[
 C_\lambda(N):=C_{1/2,\lambda}(N),\qquad
 \Gamma_{N,\lambda}:=\Gamma_{N,1/2,\lambda},\qquad
 m_{1/2}=\frac8{3\pi}.
\]
Here
\[
 F_{1/2}(x)=(1+x)^{3/2}+(1-x)^{3/2}-2-\frac34x^2
\]
and Proposition~\ref{prop:exact} reads
\begin{equation}\label{eq:half-exact}
 \Gamma_{N,\lambda}
 =\frac{1}{[N(N+1)]^{3/2}}
 \left\{
 \sum_{k=1}^Nk^3F_{1/2}(1/k)C_\lambda(k)
 +\frac{3\lambda}{4}\sum_{k=1}^N\frac{k}{k+\lambda}C_\lambda(k)
 \right\}.
\end{equation}
The exact half-moment formula gives the shape-dependent comparison bounds needed in
Paper~II.

\begin{corollary}[Half-moment shape reserves]
For every $N\ge1$ and $\lambda\ge0$,
\[
 \partial_\lambda\Gamma_{N,\lambda}
 >\frac{2}{\pi[N(N+1)]^{3/2}}
 \sum_{k=1}^N\left(\frac{k}{k+\lambda}\right)^2,
\]
and
\[
 \Gamma_{N,\lambda}
 >\frac{2\lambda}{\pi[N(N+1)]^{3/2}}
 \sum_{k=1}^N\frac{k}{k+\lambda}.
\]
\end{corollary}

\begin{proof}
Both assertions follow from Proposition~\ref{prop:shape-derivative-reserve}
and its proof, using $m_{1/2}=8/(3\pi)$.
\end{proof}

\subsection{The square half-moment benchmark}
At $\lambda=0$, Abreu and Patil obtain a sharper square-specific description
than is needed for the two-parameter theory developed here.  With
\[
 \ell_N:=\log N+\gamma+6\log2,\qquad
 \mathsf H_N:=\sum_{k=1}^N\frac1k,
\]
their non-asymptotic estimate \cite[Theorem~1]{AbreuPatil} gives
\begin{equation*}
 0<\frac{\ell_N-10/3}{8\pi[N(N+1)]^{3/2}}
 <\Gamma_{N,0}
 <\frac{\mathsf H_N+6\log2-10/3}{8\pi[N(N+1)]^{3/2}}.
\end{equation*}
They also prove complete asymptotic expansions; in particular,
\begin{align}
 \Gamma_{N,0}
 &=\frac{\ell_N-10/3}{8\pi N^3}
   +O\!\left(\frac{\log N}{N^4}\right),
 \label{eq:AP-square-dec-asymptotic}\\
 C_0(N)-\frac8{3\pi}
 &=\frac{\ell_N-17/6}{16\pi N^2}
   +O\!\left(\frac{\log N}{N^4}\right).\notag
\end{align}
The square case is therefore known more precisely than is needed here.  The next theorem instead keeps the shape parameter in a bounded interval and identifies the coefficient that vanishes at $\lambda=1/2$.

\subsection{A bounded-shape two-term expansion}

\begin{theorem}[Uniform two-term half-moment expansion]
For every $\Lambda<\infty$, uniformly for $0\le\lambda\le\Lambda$,
\begin{align}
 \Gamma_{N,\lambda}
 &=\frac{2\lambda}{\pi N^2}
 +\frac{1-4\lambda^2}{8\pi}\frac{\log N}{N^3}
 +O_\Lambda(N^{-3}),
 \label{eq:half-refined-dec}\\
 C_\lambda(N)-\frac8{3\pi}
 &=\frac{2\lambda}{\pi N}
 +\frac{1-4\lambda^2}{16\pi}\frac{\log N}{N^2}
 +O_\Lambda(N^{-2}).
 \label{eq:half-refined-level}
\end{align}
\end{theorem}

\begin{proof}
\smallskip\noindent\emph{First-order control.}
Before extracting the logarithmic coefficient, we need the first correction
to $C_\lambda(k)$.  Lemma~\ref{lem:shape-uniform} supplies a bound
$C_\lambda(k)\le M_\Lambda$.  Since $k^3F_{1/2}(1/k)=O(k^{-1})$, the first sum in
\eqref{eq:half-exact} is $O_\Lambda(\log N)$, while the second is $O_\Lambda(N)$.  Hence
$\Gamma_{N,\lambda}=O_\Lambda(N^{-2})$.  By Corollary~\ref{cor:concave-all-shapes},
$\Gamma_{N,\lambda}>0$, and therefore
\[
 0\le C_\lambda(N)-m_{1/2}
 =\sum_{j=N}^\infty\Gamma_{j,\lambda}=O_\Lambda(N^{-1}).
\]
Consequently,
\[
 \frac1N\sum_{k=1}^N\frac{k}{k+\lambda}C_\lambda(k)
 =m_{1/2}+O_\Lambda\!\left(\frac{\log N}{N}\right),
\]
and substitution in \eqref{eq:half-exact}, followed by telescoping, gives
\begin{align}
 \Gamma_{N,\lambda}
 &=\frac{2\lambda}{\pi N^2}
 +O_\Lambda\!\left(\frac{\log N}{N^3}\right),\notag
 \\
 C_\lambda(N)
 &=m_{1/2}+\frac{2\lambda}{\pi N}
 +O_\Lambda\!\left(\frac{\log N}{N^2}\right).
 \label{eq:half-first-level}
\end{align}
\smallskip\noindent\emph{The square-source sum.}
The square source contributes one logarithm.  With the first-order estimate
available, the binomial expansion gives, for $k\ge2$,
\begin{align*}
 k^3F_{1/2}(1/k)&=\frac{3}{64k}+O(k^{-3}),\\
 k^3F_{1/2}(1/k)C_\lambda(k)
 &=\frac1{8\pi k}+O_\Lambda(k^{-2}).
\end{align*}
The finitely many omitted initial terms are absorbed into the uniform
$O_\Lambda(1)$ constants below.  Consequently,
\begin{equation}\label{eq:half-sum-one}
 \sum_{k=1}^Nk^3F_{1/2}(1/k)C_\lambda(k)
 =\frac{\log N}{8\pi}+O_\Lambda(1).
\end{equation}
\smallskip\noindent\emph{The shape-source sum.}
The shape source contributes a second logarithm through the first correction
to $C_\lambda(k)$ and the expansion of $k/(k+\lambda)$.  For $k\ge2$,
\[
 \frac{k}{k+\lambda}=1-\frac\lambda k+O_\Lambda(k^{-2}),
\]
and \eqref{eq:half-first-level} imply
\begin{equation}\label{eq:half-sum-two}
 \sum_{k=1}^N\frac{k}{k+\lambda}C_\lambda(k)
 =m_{1/2}N-\frac{2\lambda}{3\pi}\log N+O_\Lambda(1).
\end{equation}
Indeed, the coefficient of $k^{-1}$ is
$2\lambda/\pi-m_{1/2}\lambda=-2\lambda/(3\pi)$.

\smallskip\noindent\emph{Combining the two sums.}
The coefficient $1-4\lambda^2$ now comes from adding the logarithmic term of
the square source to the logarithmic correction in the shape source.
Substitution of \eqref{eq:half-sum-one} and
\eqref{eq:half-sum-two} into \eqref{eq:half-exact}, together with
$[N(N+1)]^{3/2}=N^3(1+O(N^{-1}))$, proves
\eqref{eq:half-refined-dec}.  Summing from $N$ to infinity and using
\[
 \sum_{j=N}^\infty j^{-2}=N^{-1}+O(N^{-2}),\qquad
 \sum_{j=N}^\infty\frac{\log j}{j^3}
 =\frac{\log N}{2N^2}+O(N^{-2})
\]
proves \eqref{eq:half-refined-level}.
\end{proof}

Two different effects occur at the half moment.  The shrinking-shape transition
has scale $\lambda\asymp(\log N)/N$ because the square and shape sources first
have comparable size there.  Separately, at the fixed shape $\lambda=1/2$ the
logarithmic coefficient in \eqref{eq:half-refined-dec} vanishes.  This second
cancellation belongs to the bounded-shape expansion and is not the
shrinking-shape transition.

\begin{remark}[Compatibility with the square literature]
At $\lambda=0$, the leading logarithmic decrement in
\eqref{eq:half-refined-dec} agrees with the term already implicit in the
positive-form proof of \cite{HutnikComplexOriginal} and with the sharper
Abreu--Patil expansion \eqref{eq:AP-square-dec-asymptotic}.  The new content
of the uniform expansion above is not a further square refinement: it
is the bounded-shape uniform law, including the explicit coefficient
$1-4\lambda^2$ and its cancellation at $\lambda=1/2$.  This distinction is
useful in Paper~II, where the square endpoint can use the sharper
Abreu--Patil bound while the positive-shape and derivative reserves must come
from the present two-parameter theory.
\end{remark}

\section{Conclusion}

The two-source formula gives a single finite-dimensional description across
all moment orders.  The square source is present at $\lambda=0$ and changes
sign at $s=1$ and $s=2$, whereas the shape source is nonnegative and is
strictly positive for $\lambda>0$.  In addition to the square sign diagram,
the full decrement satisfies, for every $s>0$ and every $N\ge1$,
\[
 \partial_\lambda\Gamma_{N,s,\lambda}>0.
\]
Thus, increasing the Laguerre shape always increases the dimension decrement,
including the convex range where the square decrement is negative.

The size of the shape perturbation needed to affect the leading square
correction depends on $s$.  The critical scale is $N^{-2s}$ for
$0<s<1/2$, $(\log N)/N$ at $s=1/2$, and $N^{-1}$ for $s>1/2$.  The
logarithm at the half moment comes from the harmonic leading term of the
square source.  For integer shape this is also the case with the direct
interpretation as the expected normalized nuclear norm of a complex Gaussian
matrix.

For $1<s<2$, the square and shape sources have opposite signs.  The level and
decrement therefore have separate transition points.  Both are unique for
every finite $N$, and
\[
 \lambda^{\mathrm{lev}}_{N,s}
 =\frac{\tau_s^*}{2N}+o(N^{-2}),
 \qquad
 \lambda^{\mathrm{dec}}_{N,s}
 =\frac{\tau_s^*}{N}-\frac{\tau_s^*}{2N^2}+o(N^{-2}).
\]
For all sufficiently large $N$, the normalized moment lies above its
Marchenko--Pastur limit between these two values but still increases when the
dimension is raised by one.  Hence the side from which the limit is approached
and the local direction of dimension monotonicity are distinct finite-dimensional
properties.  For $s>2$, both sources are nonnegative for every $\lambda\ge0$,
so no analogous sign transition occurs.

The closed LUE dimension recurrence is used throughout at finite $N$.
Marchenko--Pastur convergence enters only when the limiting moment and the
coefficients of the large-$N$ expansions are identified.  The shape
monotonicity and the exact crossing bounds themselves are finite-dimensional
consequences of the recurrence and positive association.

At $s=1/2$, the positive-shape and derivative estimates above are the unitary
inputs used in Paper~II~\cite{HutnikLOE}.  At square shape Paper~II uses the
sharper Abreu--Patil estimate instead.  For bounded positive shape, the
two-term expansion here supplies the unitary logarithmic coefficient that is
combined there with the orthogonal correction.

Beyond the half moment treated in Paper~II, the same full moment-order
question can be asked for the orthogonal and symplectic Laguerre ensembles.
Their dimension recurrences are coupled rather than closed, so extending the
present two-source analysis to general $s$ would require a different argument.

\section*{Acknowledgement}
This work was supported by the Slovak Research and Development Agency under
contract No.~APVV-25-0144. During the preparation of this work, the author used ChatGPT (OpenAI) as an auxiliary tool for literature searches, organization and language editing of the exposition, and for checking and simplifying intermediate calculations, asymptotic expansions, and parts of proofs. The author takes full responsibility for the content.

\end{document}